\documentclass[english,10pt]{amsart}
\usepackage{amsmath, amssymb, amsthm, enumerate}
\usepackage[all]{xy}
\usepackage[activeacute]{babel}
\usepackage{appendix}

\newtheorem{thm}[equation]{Theorem} 
\newtheorem{lem}[equation]{Lemma}
\newtheorem{prop}[equation]{Proposition}
\newtheorem{cor}[equation]{Corollary}

\newtheorem*{conj}{Voevodsky's Conjecture 8}

\theoremstyle{remark}
\newtheorem{rmk}[equation]{Remark}

\newtheorem{notation}[equation]{Notation}
\newtheorem{substuff}{Remark}[equation]

\newtheorem{subrem}[substuff]{Remark}

\theoremstyle{definition}
\newtheorem{defn}[equation]{Definition}
\newtheorem{exam}[equation]{Example}
\newtheorem{variant}[equation]{Variant}

\newcommand{\Hom}{\mathrm{Hom}}

\newcommand{\Proj}{\mathbb P}
\newcommand{\spherespectrum}{\mathbf{1}}

\newcommand{\stablehomotopy}{\mathcal{SH}}

\newcommand{\fnteff}{\mathrm{eff}}
\newcommand{\eff}{\stablehomotopy^{\fnteff}}

\def\lra{\longrightarrow}
\def\map#1{\, {\buildrel #1 \over \lra}\, }
\def\lmap#1{\,{\buildrel #1 \over \longleftarrow}\,}
\newcommand{\Spec}{\operatorname{Spec}}
\newcommand{\Z}{\mathbb{Z}}
\newcommand{\Q}{\mathbb{Q}}
\newcommand{\C}{\mathbb{C}}
\newcommand{\Pinfty}{\mathbb P^{\infty}} 
\newcommand{\mathdot}{{\mathbf{\scriptscriptstyle\bullet}}}
\def\oo{\otimes}
\def\gr{\text{gr}}
\newcommand{\SH}{\stablehomotopy}
\def\threerightarrows{\ \vbox to8pt{\hbox{$\rightarrow$}\kern-8.0pt
     \hbox{$\rightarrow$}\kern-8.0pt\hbox{$\rightarrow$}\vskip-3.6pt}\ }
\def\fourrightarrows{~\lower3.5pt\vbox{\hbox{$\rightrightarrows$}\vskip-5.2pt
    \hbox{$\rightrightarrows$}}~}

\numberwithin{equation}{section}

\begin{document}


\title{Slices of co-operations for $KGL$}

\date{\today}

\author{Pablo Pelaez}
\address{Department of Mathematics, Rutgers University, 
New Brunswick, NJ U.S.A.}
\email{pablo.pelaez@gmail.com}

\author{Charles Weibel}
\address{Department of Mathematics, Rutgers University, 
New Brunswick, NJ U.S.A.}
\email{weibel@math.rutgers.edu}





\begin{abstract}
We verify a conjecture of Voevodsky, concerning the
slices of co-operations in motivic $K$-theory.
\end{abstract}

\maketitle


\section*{Introduction}

Fix a finite-dimensional Noetherian separated base scheme $S$, and 
consider the motivic stable homotopy category $\SH(S)$ as
defined in \cite{V-ICM}. We write $KGL$ for the motivic spectrum 
representing homotopy invariant $K$-theory in $\SH(S)$. 
In this paper, we study Voevodsky's
conjecture on slices of co-operations for $KGL$
(Conjecture 8 in \cite{Vslice}), which describes the motivic slices
of the motivic spectra $KGL\wedge\cdots\wedge KGL$.
We verify the conjecture when $S$ is smooth over a perfect field
(see Theorem \ref{thm:perfect} below).

To describe the conjecture, recall that for any motivic spectrum $E$, 
Voevodsky introduced a natural ``slice'' tower 
$\cdots\to f_{q+1}E\to f_qE\to\cdots$ of motivic spectra, and defined
triangulated ``slice'' functors $s_q$ fitting into cofibration
sequences $f_{q+1}E\to f_qE\to s_qE$.
Roughly, $\{ f_qE\}$ is the analogue of the Postnikov tower in topology;
$s_q$ is the analogue of the functor $X\mapsto K(\pi_qX,q)$.

If $E$ is a topological ring spectrum, the ring $E_*E=\pi_*(E\wedge E)$
is called the {\it ring of co-operations} for $E$; the name comes
from \cite{AHS}, where $\pi_*(KU\wedge KU)$ is worked out
(see Section \ref{sec:binomial} below). 
The title of this paper comes from viewing the motivic spectrum 
$KGL\wedge KGL$ as giving rise to co-operations for $K$-theory. 
The easy part of Voevodsky's Conjecture 8 says that 
there is an isomorphism
\[
s_q(KGL\wedge KGL) \cong (T^q\wedge H\Z)\oo \pi_{2q}(KU\wedge KU).
\] 
Here $H\Z$ is the motivic Eilenberg-Mac\,Lane spectrum in $\SH(S)$, 
$T$ is the motivic space represented by the pointed projective line,
and the tensor product of a spectrum with an abelian group 
has its usual meaning (see \ref{Notation}).

Since a ring spectrum $E$ is a monoid object in spectra, co-operations
fit into a cosimplicial spectrum $N^\mathdot E = E^{\wedge\mathdot+1}$.
This is a general construction: in any category with product $\wedge$, 
a monoid $E$ determines a triple $\top X=X\wedge E$ and an 
augmented cosimplicial object $n\mapsto X\wedge E^{\wedge n+1}$
for any object $X$;  the cofaces are given by 
the unit $\spherespectrum\to E$ and
the codegeneracies are given by the product 
$E\wedge E\to E$; see \cite[8.6.4]{WH}.  
We write $N^\mathdot E$ when $X=\spherespectrum$, so
$N^{n}E = E^{\wedge n+1}$\!.


\begin{conj}[Slices of co-operations for $KGL$] \label{conjecture8}
As cosimplicial\\ motivic spectra, the $q^{th}$ slice satisfies 
\[ s_q N^\mathdot(KGL) \cong (T^q\wedge H\Z)\oo\pi_{2q} N^\mathdot(KU).   
\] 
\end{conj}

One reformulation of this conjecture is to take the direct sum
over $q$ and use the fact that $s_*E=\oplus s_q(E)$ is a graded 
motivic ring spectrum (see \cite[3.6.13]{Pelaez}). 
It is convenient to adopt the notation that $E\oo A_*$ denotes
$\oplus_q (T^q\wedge E)\oo A_q$ for a motivic spectrum $E$
and a graded ring $A_*$. In this notation, we prove:

\begin{thm}\label{thm:perfect}
Assume that $S$ is smooth over a perfect field. Then
there is an isomorphism of cosimplicial ring spectra
\[
s_* N^\mathdot(KGL) \cong H\Z\oo \pi_{2*}N^\mathdot(KU).
\]
\end{thm}

\begin{subrem}
If $S$ is over a field of characteristic~0, it seems likely that
Theorem \ref{thm:perfect} would follow from the work of 
Spitzweck \cite{Sp} on Landweber exact spectra. Such an approach
would depend heavily on the Hopkins-Morel-Hoyois theorem \cite{Hoyois}.
\end{subrem}

Conjecture 8 is intertwined with Voevodsky's conjectures 1, 7 and 10
in \cite{Vslice}, that $H\Z \leftarrow\spherespectrum\to KGL$ induces
isomorphisms
\[
H\Z \cong s_0(H\Z) \lmap{\cong} s_0(\spherespectrum)\map{\cong} s_0(KGL),
\]
and thus that 
$\Hom_{\SH(S)}(s_0(\spherespectrum),s_0(\spherespectrum))\cong H^0(S,\Z)$.
These are known to hold when the base $S$ is smooth over a perfect field,
or singular over a field of characteristic~0, by the work of 
Voevodsky and Levine (see \cite[10.5.1]{L} and \cite[11.3.6]{L}).

Here is our main result, which evidently implies Theorem \ref{thm:perfect}.

\begin{thm}\label{thm:main}
Let $S$ be a finite-dimensional separated Noetherian scheme. Then 
\\
(a) there are isomorphisms for all $n\ge0$:
\[
s_0(KGL)\oo \pi_{2*}(KU^{\wedge n}) \map{\cong} s_*(KGL^{\wedge n}).
\]
These isomorphisms commute with all of the coface and codegeneracy 
operators exept possibly $\partial^0$ and $\sigma^0$. \\
(b) Assume in addition that
$s_0(\spherespectrum)\to s_0(KGL)$ is an isomorphism in $\SH(S)$,
and that $\Hom_{\SH(S)}(s_0(\spherespectrum),s_0(\spherespectrum))$
is torsionfree.
Then the maps in (a) are the components of an isomorphism of 
graded cosimplicial motivic ring spectra:
\[
s_0(KGL)\oo \pi_{2*}N^\mathdot(KU)
\map{\cong} s_* N^\mathdot(KGL). 
\]
\end{thm}

The case $n=0$ of Theorem \ref{thm:main}(a), that
$s_0(KGL)\oo\pi_{2*}KU\cong s_* KGL$, is immediate from the 
periodicity isomorphism $T\wedge KGL\cong KGL$ defining
the motivic spectrum $KGL$ and the formula
$\pi_{2*}KU\cong\Z[u,u^{-1}]$.  The need to pass to slices
is clear at this stage, because $\pi_{2n,n}KGL\cong K_n(S)$
for $S$ smooth over a perfect field.

The left side of Theorem \ref{thm:main} is algebraic in nature,
involving only the cosimplicial ring $\pi_{2*}(KU^{\wedge n+1})$
and $H=KU_0(KU)$.
In fact, $\pi_{2*}(KU^{\wedge n+1})$ is the cobar construction 
$C^\mathdot_\Gamma(R,R)\cong KU_*\oo H^{\oo n}$
for the Hopf algebroid $(R,\Gamma)=(KU_*,KU_*KU)$.
We devote the first four sections to an analysis of this algebra,
focussing on the rings $F=KU_0(\C\Pinfty)$ and $H$.
Much of this material is well known, and due to Frank Adams.

If $E$ is an oriented motivic spectrum, the projective bundle 
theorem says that $E\wedge\Pinfty\cong E\oo F$ and hence that
$E\wedge{\Pinfty}^{\wedge n}\cong E\oo F^{\oo n}$.
Using this, we establish a toy version of Theorem \ref{thm:main}
in Propositions \ref{toy.global} and \ref{toy.slice} that, 
as cosimplicial spectra, $KGL\wedge{\Pinfty}^{\wedge\mathdot+1}$
is $KGL\oo F^{\oo\mathdot+1}$ (the cobar construction on $F$) and
\begin{equation}\label{eq:s_q(K.Pdot)}
s_*(KGL\wedge{\Pinfty}^{\wedge\mathdot+1}) \cong 
s_*(KGL)\oo F^{\oo\mathdot+1} \cong
s_0(KGL)\oo KU_*({\C\Pinfty}^{\wedge n}).
\end{equation}

Using a theorem of Snaith (\cite{GS} and \cite{SO}), we use the
toy model to show that $KGL\wedge KGL\cong KGL\oo H$ 
and more generally (in \ref{K.toy}) that 
\begin{equation}\label{eq:K.Kdot}
KGL^{\wedge\mathdot+2} = \; KGL\oo H^{\oo\mathdot+1}
\end{equation}
as cosimplicial spectra.  Taking slices in \eqref{eq:K.Kdot}
gives the isomorphisms in Theorem \ref{thm:main}(a), and
(with a little decoding) also proves compatibility  with 
every coface and codegeneracy map except for $\partial^0$ and $\sigma^0$.
This proves Theorem \ref{thm:main}(a).

Compatibility with the coface maps $\partial^0$ is established in
Proposition \ref{partial0}, using the extra hypothesis that
$s_0(\spherespectrum)\cong s_0(KGL)$,
and compatibility with the codegeneracy
map $\sigma^0$ is established in Lemma \ref{sigma0.t},
using the torsionfree hypothesis.
This proves Theorem \ref{thm:main}(b).

\smallskip
The paper is organized as follows. In Sections \ref{sec:binomial}
and \ref{sec:modp} 
we introduce the binomial rings $F$ and $H$. As noted in 
Remark \ref{rem:F-H}, $KU_*(\C\Pinfty)=KU_*\oo F$ and
$KU_*(KU)=KU_*\oo H$.  In Section \ref{sec:algebroid}
we quickly review Hopf algebroids and the algebroid structure on
$(KU_*,KU_*KU)$. In Section \ref{sec:KU}, we recall how
$\pi_* N^{\mathdot}KU$ is the cobar complex for this algebroid,
by showing that $KU\wedge{\Pinfty}^{\wedge\mathdot+1}$ is
$KU\oo C^\mathdot(F,F)$, where $C^{\mathdot}(F,F)$ is the 
cobar complex of the Hopf algebra $F$.
Most of this material is due (at least in spirit) to Adams.

In Section \ref{sec:KGL} we show that 
$KGL\wedge{\Pinfty}^{\wedge\mathdot+1}$ is
$KGL\oo C^\mathdot(F,F)$, by mimicking the development of Section \ref{sec:KU}.
In Section \ref{sec:slice} we show that the slice functors commute 
with direct sums, and deduce \eqref{eq:s_q(K.Pdot)}. 
Section \ref{sec:E.K} establishes \eqref{eq:K.Kdot}, and 
Theorem \ref{thm:main} is established in Section \ref{sec:final}.

\begin{notation}\label{Notation}
For any graded abelian group $A$ and motivic spectrum $E$, we can
form a motivic spectrum $E\oo A$, as follows. If $A$ is free with a
basis of elements $a_i$ in degrees $d_i$, then $E\oo A$ is the wedge
of the $E\wedge T^{d_i}$, and we may regard $a_i$ as a map from
$E\wedge T^{d_i}$ to $E\oo A$. For general $A$, choose a free graded
resolution $0\to P_1\to P_0\to A\to 0$ and define $E\oo A$ to be the
cofiber of $E\oo P_1\to E\oo P_0$; Shanuel's Lemma implies that this
is independent of the choice of resolution, up to isomorphism.
The choice of a lift of a homomorphism $A\to B$ is unique up to chain
homotopy, so it yields a map $E\oo A\to E\oo  B$ unique up to homotopy.
Given homomorphisms $A\to B\to C$, this yields a homotopy between 
the composition $E\oo A\to E\oo B\to E\oo C$ and $E\oo A\to E\oo C$.
That is, this construction gives a lax functor from abelian groups to
strict motivic spectra over $E$, 
and a functor to motivic spectra over $E$.

There is a natural associative map
$(\spherespectrum\oo A)\wedge(\spherespectrum\oo B)\to
\spherespectrum\oo A\oo B$, at least if $A$ and $B$ have no summands
$\Z/2$, $\Z/3$ or $\Z/4$ \cite[IV.2.8]{WK};
it is an isomorphism if $A$ and $B$ are free abelian groups. Thus
if $E$ is a ring spectrum and $A$ is a ring, the composition
$(E\oo A)\wedge(E\oo A)\to E\oo A$ makes $E\oo A$ into a ring spectrum.
\end{notation}

\newpage
\section{Universal binomial rings}\label{sec:binomial}

Recall that a {\it binomial ring} is a subring $R$ of a $\Q$-algebra 
which is closed under the operations $r\mapsto\binom{r}{n}$.
It is a $\lambda$-ring with operations $\lambda^n(r) = {\binom rn}$.

For example, consider the subring $F$ of $\Q[t]$ consisting of
{\it numerical polynomials} -- polynomials $f(t)$ with $f(n)\in\Z$
for all integers $n\gg0$. It is well known that $F$ is free
as an abelian group, and that the $\alpha_n=\binom tn$ form a basis.
It is not hard to verify the formula that
\[
t{\binom tn} = n{\binom tn} + (n+1){\binom t{n+1}}.
\]
The general ring structure of $F$ is determined by the combinatorial identity:
\begin{equation}\label{alpha*alpha}
\alpha_i * \alpha_j = \sum_{k\le i+j}\binom{k}{k-i,k-j,i+j-k}\alpha_k.
\end{equation}
Here $\binom{k}{a,b,c}$ denotes $k!/a!b!c!$.
(To derive \eqref{alpha*alpha}, note that the left side counts pairs of 
subsets of a set with $t$ elements. If the union of an $i$-element set and a
$j$-element set has $k$ elements, the sets intersect in $i+j-k$ elements.)

The universal polynomials for $\lambda^m(\lambda^n(r))$ show that 
the numerical polynomials form a binomial ring.
In fact, $F$ is the free binomial ring on one generator $t$:
if $R$ is binomial and $r\in R$ the canonical extension of the 
universal ring map $\Z[t]\to R$ to $F\to R\otimes\Q$ factors
uniquely through a map $F\to R$.

\begin{defn}\label{def:H}
Let $H$ denote the localization $F[1/t]$ of the ring of numerical
polynomials; it is a subring of $\Q[t,1/t]$. 
\end{defn}

Here is a useful criterion for membership in $H$.

\begin{lem}
$H=F[1/t]$ is the ring of all $f(t)\in\Q[t,1/t]$ such that for any
positive integer $a$ we have $f(a)\in\Z[1/a]$.
\end{lem}

\begin{proof} (Cf.\,\cite[5.3]{AHS})
Multiplying $f$ by a suitable power of $t$, we may assume that
$f(t)\in t^\nu\Q[t]$, where $\nu$ is the highest exponent of any
prime occurring in the denominators of the coefficients of $f$.
It suffices to show that $f$ is a numerical polynomial.
Fix $a>0$ and let $p$ be a prime. If $p|a$ then $p$ does not appear
in the denominator by construction; if $p\nmid a$ then $p$ does not
appear in the denominator of $f(a)$ by hypothesis. Hence
$f(a)\in\Z$, as desired.
\end{proof}


Recall that $\Q[t,1/t]$ is a Hopf algebra with $\Delta(t)=t_1t_2$.
The usual proof \cite[I.7.3]{Hart}
that the functions $\binom tn$ form a basis of the ring
$F$ of numerical polynomials $f(t)$ is easily modified to show that 
the functions $\binom{t_1}{m}\binom{t_2}{n}$ form a $\Z$-basis of the ring
of numerical polynomials $f(t_1,t_2)$ in $\Q[t_1,t_2]$. 
Identifying $\Q[t_1,t_2]$ with $\Q[t]\oo\Q[t]$, we obtain a canonical 
isomorphism between the subring $F\otimes F$ of $\Q[t]\oo\Q[t]$
and the ring of numerical polynomials in $\Q[t_1,t_2]$.

\begin{thm}\label{H=Hopf}
(a) $H$ is a Hopf subalgebra of $\Q[t,1/t]$.
\\ (b) $H$ is a binomial ring; it is the free binomial ring on a unit.
\end{thm}

\begin{proof}
For (a), it suffices to show that $\Delta:\Q[t]\to \Q[t_1,t_2]$ sends
$F[1/t]\subset\Q[1/t]$ into the subring $F[1/t]\otimes F[1/t]$. 
But $\Delta$ sends $\binom tn$ to $f(t_1,t_2)=\binom{t_1t_2}{n}$, 
which is a numerical
polynomial and hence belongs to $F\otimes F$. Thus $\Delta$ sends
$t^{-k}\binom{t}{n}$ to a $\Z$-linear combination of the functions
$(t_1t_2)^{-k}\binom{t_1}{m}\binom{t_2}{n}$, which is in
$H\oo H$. 

For (b), fix $f(t)/t^m$ in $F[1/t]$, and set $g_k(t)=\lambda^k(f/t^m)$.
For each nonzero $a\in\Z$, $\Z[1/a]$ is a binomial ring, so 
$g_k(a)= \lambda^k(f(a)/a^m)$ is in $\Z[1/a]$. 
\end{proof}

\begin{thm}
As an abelian group, $H$ is free.
\end{thm}

\begin{proof}
For $m\le n$, let $F(m,n)$ denote the intersection of $H$
with the $\Q$-span of $t^m,...,t^n$ in $\Q[t,1/t]$. Then the proof
of \cite[2.2]{AC} goes through to prove that $F(m,n)\cong R^{1+n-m}$
and that $H$ is free abelian, given the following remark: For each $k,m,n$,
$\binom {kt}n$ is a numerical polynomial, so
$(kt)^{-m}{\binom {kt}n}$ is certainly in $F[1/t][1/k]$.
\end{proof}

\begin{rmk}\label{rem:F-H}
The ring  $KU_{\ast}\oo F$ is $KU_*(\C\Pinfty)$, and 
$KU_{\ast}\oo H$ is isomorphic to the ring $KU_0(KU)$.
These observations follow from \cite[II.3]{A} and
\cite[2.3, 4.1, 5.3]{AHS}.

A $\Z$-basis of $KU_*(KU)$ was given in \cite[Cor.\,13]{J}.
\end{rmk}

\medskip
Since we will be interested in the algebras $R\otimes F[1/t]$
over different base rings $R$, we now give a slightly different
presentation of $F_R=R\otimes F$ and $F_R[1/t]$.
As an $R$-module, $F_R$ is free with countable basis 
$\{ \alpha _{0}, \alpha _{1}, \ldots \}$,
and we are given an $R$-module map $T:F_R\to F_R$
(multiplication by $t$), defined by
\begin{equation}\label{eq:T}
T(\alpha_n)= n\alpha_n + (n+1)\alpha_{n+1}. 
\end{equation}
Note that $T(\alpha_0)=\alpha_1$.
The localization $F_R[1/t]$ is the colimit of the system 
\begin{equation}\label{seq:F}
F_R\map{T}F_R\map{T}F_R\map{T}\cdots.
\end{equation}
To describe it, we introduce the bookkeeping index $t^{-j}$ 
to indicate the $j^{th}$ term in this sequence.
By \cite[2.6.8]{WH}, 
there is a short exact sequence of $R$-modules:
\[
0\to \bigoplus_{j=0}^{\infty}F_R\;t^{-j} \map{\Phi}
     \bigoplus_{j=0}^{\infty}F_R\;t^{-j} \to F_R[1/t] \to 0,
\]
where $\Phi(\alpha_{n}t^{-j})=T(\alpha_n)t^{-j-1}-\alpha_n t^{-j}$.
Since $\Phi$ is $t^{-1}$-linear, we may regard $\Phi$ as
an endomorphism of the free $R[t^{-1}]$-module 
$F_R[t^{-1}]=\oplus F_R\; t^{-n}$ with basis $\{\alpha_m\}$,
with 
\[
\Phi(\alpha_n)=T(\alpha_n)t^{-1}-\alpha_n
             = (n+1)t^{-1}\alpha_{n+1}+(nt^{-1}-1)\alpha_n.
\]
By abuse of notation, we write $\alpha_n$ for the image in $F_R[1/t]$ of the
basis element $\alpha_n$ of $F_R$.
Thus $F_R[1/t]$ may be presented as the $R[1/t]$-module 
with generators $\alpha_n$ and relations
\begin{equation}\label{eq:rel}
(1-nt^{-1})\alpha_n=(n+1)t^{-1}\alpha_{n+1}, \quad n\ge0.
\end{equation}

Notice that $\alpha_{0}=t^{-1}\alpha_{1}$. It is not hard to verify
directly, beginning with \eqref{seq:F}, that
multiplication by $t^{-1}$ is an isomorphism on $F_R[1/t]$. 
It also follows directly from the ring structure on $F$ and the 
identification of the colimit with $F_R[1/t]$.

If $R$ contains no $\Z$-torsion, so that $R\subseteq R\otimes\Q$,
it is easy to see that $F_R[1/t]$ embeds in 
$F_{R\otimes\Q}[1/t]=R\otimes\Q[t,t^{-1}]$,
as we saw at the beginning of this section.
The presentation of $R\oo H$ as an $R$-module looks different when 
$R$ has $\Z$-torsion, as we shall see in Section \ref{sec:modp}.

\begin{exam}\label{R=Q}
(See \cite[2.2]{AHS})
If $R$ contains $\Q$, the presentation \eqref{eq:rel}
shows that $F_R[1/t]$ is the $R[t^{-1}]$-module with generators 
$\{\alpha_{0},\alpha_{1},\ldots\}$ modulo the relations:
\[
t^{-n}\alpha_{n+1}
=\frac{1}{(n+1)!}(1-t^{-1})(1-2t^{-1})\cdots(1-nt^{-1})\alpha_1.
\]
That is, $F_R[1/t]\cong R[t,t^{-1}]$ on 
generator $\alpha_0$ with the relations 
\[
\alpha_{n} = \binom{t}{n}\alpha_0.
\]
\end{exam}

\begin{variant}\label{variant:v=ut}
Adams uses a variant of the above construction. Fix a unit $u$ of $R$ 
and set $v=ut$, so $t=v/u$. We consider $B=uT$ to be multiplication by $v$
with \eqref{eq:T} replaced by $B(\beta_n)=nu\beta_n + (n+1)\beta_{n+1}$,
where $\beta_n=u^n\alpha_n$.
Replacing the bookkeeping index $t^{-1}$ by $v^{-1}$, \eqref{eq:rel}
becomes $(1-nu/v)\beta_n=(n+1)v^{-1}\beta_{n+1}$, 
or $(n+1)\beta_{n+1}=(t-n)u\beta_n$, and we recover
\[
\beta_{n} = u^n{\binom{v/u}n}\beta_0 = u^n{\binom tn}\beta_0.
\]
These $\beta_n$ are the elements described by Adams in \cite[II.13.7]{A} 
as the generators of $KU_*(KU)$,
regarded as a submodule of $KU_*(KU)\oo\Q=\Q[u,1/u,v,1/v]$.
\end{variant}

\bigskip
\section{The $\ell$-primary decomposition}\label{sec:modp}

In this section, we suppose that $R$ is an algebra over $\Z/\ell^\nu$
and give a basis for the colimit $H_R=F_R[1/t]$ 
of the sequence 
$F_R\map{T}F_R\map{T}F_R\map{T}\cdots$ of \eqref{seq:F}. 

Recall from Section \ref{sec:binomial} that the elements of $F$
are polynomial functions $\mathbb{N}\to\Z$, and that the
$\alpha_n=\binom tn$ form a basis of $F$.

\begin{lem}\label{lem:modular}
For each $q\in\Z$, $F\oo\Z/q$ embeds in the ring of all functions
$\mathbb{N}\to\Z/q$.
\end{lem}

\begin{proof}
If $f\in F$ satisfies $f(a)\equiv0\pmod{q}$ for every $a>0$
then $h(t)=f(t)/q$ is a numerical polynomial, 
and $f(t)=q\,h(t)$ is in $qF$.
\end{proof}

\begin{exam}\label{R=Z/2}
Suppose that $\Z/2\subseteq R$. Then the relations
$\alpha_{2k}=t^{-1}\alpha_{2k+1}$ and $(1-t^{-1})\alpha_{2k+1}=0$ imply that
$H_R=F_R[1/t]$ is the free $R$-module with basis 
$\{\alpha_{2k+1}, k\ge0\}$, with $\alpha_{2k}=\alpha_{2k+1}$.

In fact, $F/2F$ and $H/2H$ are Boolean rings by Lemma \ref{lem:modular}. 
In particular, each $\alpha_n$ is idempotent in the ring $F/2F$,
including $t=\alpha_1$, and $H/2H=t\,F/2F$ is a factor ring of $F/2F$. 
In addition, if $m\le2^r-1$ then
\eqref{alpha*alpha} implies that $\alpha_m\alpha_{2^r-1}=\alpha_{2^r-1}$.
\end{exam}

When $R$ is a $\Z/\ell^\nu$-algebra, $F_R$ has a similar block decomposition.
To prepare for it, 
let $L'$ denote the 
free $R$-module on basis $\alpha_0, \alpha_1,...,\alpha_{\ell-1}$,
and let $L$ denote the submodule of $L'$ on $\alpha_1,...,\alpha_{\ell-1}$. 
Following \eqref{eq:T}, 
we define maps $b_k:L'\to L'$ by 
\begin{equation}\label{eq:b_k}
b_{k}(\alpha_i)=(k\ell+i)\alpha_i + (k\ell+i+1)\alpha_{i+1},
	 \quad i=0,...,\ell-2
\end{equation}
and $b_k(\alpha_{\ell-1})=(k\ell+\ell-1)\alpha_{\ell-1}$. 
Note that $b_k(L)\subseteq L$.

\begin{lem}\label{M[1/b]}
The maps $L\to L[1/b_k]\to L'[1/b_k]$ 
are isomorphisms for all $k$.
\end{lem}

\begin{proof}
The restriction of $b_k$ to $L$ is represented by a 
lower triangular matrix, whose determinant 
$\prod_{i=1}^{\ell-1}(k\ell+i)$ is a unit in $R$. Thus
each $b_k$ restricts to an automorphism of $L$.
Since $\ell^\nu=0$ in $R$, $(b_k)^{\nu}$ maps $\alpha_0$ into $L$.
The result is now straightforward.
\end{proof}

We can now describe the $R$-module $H_R=R\oo F_R[1/t]$ when $R=\Z/\ell$; 
the case $\ell=2$ was given in Example \ref{R=Z/2}.
The maps $\phi_k:L'\to F_R$, 
$\phi_k(\alpha_i)=\alpha_{k\ell+i}$, induce an isomorphism 
$\oplus\phi_k:\bigoplus_{k=0}^\infty L' \to F_R$ under which
the map $\oplus b_k$ is identified with the map $T$ of \eqref{eq:T}. 

\begin{cor}\label{R=Z/l}
If $\Z/\ell\subseteq R$ with $\ell$ prime, the map
$\oplus\phi_k: \bigoplus_{k=0}^\infty L \to H_R$ is an isomorphism.
Thus the elements $\alpha_n$, $n\not\equiv0\pmod{\ell}$ form a basis of $H_R$.
\end{cor}

\begin{proof}
The maps $\phi_k:L'\to F$
satisfy $\phi_k\circ\oplus b_k = T\circ\phi$. 
Hence the map $\oplus\phi_k$ induces an isomorphism between 
$\bigoplus L\cong\bigoplus L'[1/b_k]$ and $H_R=F_R[1/t]$.
\end{proof}

\begin{thm}\label{H/q.basis}
If $R$ is a $\Z/\ell^\nu$-algebra, the $R$-module map
$\oplus\phi_k:\bigoplus L \map{\cong}H_R$ is an isomorphism.
In particular,
the elements $\alpha_n$, $n\not\equiv0\pmod{\ell}$, $n>0$, form a basis.
\end{thm}

\begin{proof}
Let $F_n$ denote the free $R$-submodule of $F_R$ on basis 
$\{\alpha_{nq+i}: 0\le i<q\}$, where $q=\ell^\nu$.
Then $F_R$ is the direct sum of the $F_n$.
Since $T(\alpha_{nq-1})=(nq-1)\alpha_{nq-1}$,
$T$ sends $F_n$ into itself, and
there are isomorphisms $F_0\to F_n$,
$\alpha_i\mapsto\alpha_{nq+i}$ commuting with $T$.
Therefore $H_R=F_R[1/t]$ is isomorphic to $\bigoplus F_n[1/t]$, 
and it suffices to show that $F_0[1/t]$ is free on the $\alpha_n$ 
with $0<n<q$ and $n\not\equiv0\pmod{\ell}$.

For $k=0,...,\ell^{\nu-1}$, let $F_0^{\ge k\ell}$ denote the $R$-submodule 
of $F_0$ generated by the $\alpha_i$ with $k\ell\le i<q$. These form a 
filtration of $F_0$, and the maps $\phi_k:L'\to F_0^{\ge k\ell}$ induce 
$R$-module isomorphisms with the filtration quotients
\[ \bar\phi_k: L' \map{\cong} F_0^{\ge k\ell}/F_0^{\ge(k+1)\ell} \]
such that $T\circ\bar\phi_k = b_k\circ \phi_k$. 
By Lemma \ref{M[1/b]}, it follows that $\phi_k$ induces an isomorphism
\[
L \cong L'[1/b_k] \map{\cong} F_0^{\ge k\ell}/F_0^{\ge k\ell+\ell}[1/t].
\]
By induction on $k$, it follows that $\oplus\phi_k$ induces
an isomorphism $\bigoplus L \cong F_0[1/t]$.
\end{proof}

\section{Hopf algebroids}\label{sec:algebroid}

\medskip\noindent
Recall from \cite[A1.1.1]{R} that a Hopf algebroid $(R,\Gamma)$ is a pair of
commutative rings, with maps $\eta_L,\eta_R:R\to\Gamma$, 
$\varepsilon:\Gamma\to R$, $c:\Gamma\to\Gamma$ and
$\Delta:\Gamma\to\Gamma\oo_R\Gamma$ satisfying certain axioms,
listed in {\it loc.\ cit.}

Let $M$ be a $\Gamma$-comodule with structure map 
$M\map{\psi}M\oo_R\Gamma$. 
Recall \cite[A1.2.11]{R} that the {\it cobar complex} 
$C_\Gamma^\mathdot(M,R)$ is the cosimplicial comodule 
with $C^0=M$ and $C^n=M\oo_R\Gamma^{\oo_Rn}$, with cofaces given by 
$\psi$, $\eta_L$ and $\Delta:\Gamma\to\Gamma\oo_R\Gamma$. 
In particular, when $M=R$ (with $\psi=\eta_R$),
the cobar complex is a cosimplicial ring.  

\begin{exam}\label{triv.algebroid}
For any commutative algebra $R$, $(R,R\oo R)$ is a Hopf algebroid with
$\eta_L(r)=r\oo1$, $\eta_R(r)=1\oo r$, $c(r\oo s)=s\oo r$ and
$\Delta(r\oo s)=r\oo1\oo s$. The cobar complex $C_{R\oo R}^\mathdot(R,R)$
is the standard cosimplicial module $n\mapsto R^{\oo n+1}$.

If $\Gamma$ is a Hopf algebra over $R$, then $(R,\Gamma)$ is a Hopf
algebroid with $\eta_L=\eta_R$ the unit, $\varepsilon$ the counit,
$c$ the antipode and $\Delta$ the coproduct. In this case, the cobar
complex is classical.
\end{exam}

\begin{exam}\label{RH.algebroid}
The pair $(R_\Q,\Gamma_\Q)=(\Q[u,1/u],\Q[u,1/u,t,1/t])$ is
a Hopf algebroid with $\eta_L(u)=u\oo1$, $\eta_R(u)=c(u)=tu$,
$c(t)=1/t$, $\varepsilon(t)=1$ and $\Delta(t)=t\oo t$.

Recall from Definition \ref{def:H} that $H$ is the subalgebra 
of $\Q[t,1/t]$ generated by the $\binom tn$.  
If $R=\Z[u,1/u]$ and $\Gamma=R\oo H$, then
$(R,\Gamma)$ is a sub-Hopf algebroid. If we set $v=tu$ then
we have $\eta_R(u)=v$, $c(u)=v$ and
\[
\Delta(u)=u\oo1,\; \Delta(t)=t\oo t,\; \Delta(v)=ut\oo t = 1\oo v.
\]
Of course, $KU_*=R$ and $\pi_*(KU\wedge KU)=\Gamma$,
and the formulas given by Adams in \cite[II.13.4]{A} show that
$(R,\Gamma)$ is the original Hopf algebroid $(KU_*,KU_*KU)$.
\end{exam}



\begin{lem}\label{algebroid}
The Hopf algebroid $(R,R\oo H)$ of Example \ref{RH.algebroid} is split.
\end{lem}

\begin{proof}
The natural inclusion of the Hopf algebra $(\Z,H)$ into
$(R,R\oo H)$ is a map of algebroids \cite[A1.1.9]{R}, 
and the identity map on $R\oo H$ is an algebra map.
\end{proof}

Fix $R$ and $\Gamma\cong R\oo H$ as in Example \ref{RH.algebroid}.
Because the algebroid $(R,\Gamma)$ is split, we can say more: 
$C_\Gamma^\mathdot(M,R)=C_H^\mathdot(M,\Z)$ (see \cite[A1.2.17]{R}), 
and the latter is the usual cobar complex of $M$ as a comodule over
the Hopf algebra $H$. For $M=R$ we have the following description.

\begin{cor}\label{cobar}
The cobar complex $C_\Gamma^\mathdot(R,R)$ has $C^n=R\oo H^{\oo n}$.
The coface maps $C^n\to C^{n+1}$ are given by the units 
$\eta_L,\eta_R:R\to R\oo H=\Gamma$ of\; $\Gamma$ and $\Delta_H$.
\end{cor}

\medskip
\section{Oriented topological spectra}\label{sec:KU}

This section is intended to serve as a template for the motivic
constructions which follow, and is restricted to topological spectra.
No originality is claimed for these results. 
To emphasize our analogy, we will write $\Pinfty$ for $\C\Pinfty_+$
in this section. 

If $E$ is a topological ring spectrum, then 
combining the ring structure of $E$ with
the $H$-space structure of $\Pinfty$ we obtain a natural ring spectrum
$E\wedge \Pinfty$.  
Following Adams \cite{A} and Snaith \cite{Sn}, 
the motivation for studying $F[1/t]$ comes from the computation 
of $E\wedge KU$, where $E$ is an oriented spectrum.

Recall that a commutative topological ring spectrum $E$
with unit $\spherespectrum\map{\eta_E}E$ is said to be {\it oriented} 
if there is a map $\Pinfty\wedge\spherespectrum\map{x} S^2\wedge E$ 
whose restriction to $\C\Proj^1_+\wedge\spherespectrum$
is $1\wedge\eta_E$ (using $S^2=\C\Proj^1$). 
As in \cite[II.2]{A}, this data yields a formal group
law on $E_*[[x]]\cong E^*(\Pinfty)$, and dual elements 
$\beta_n$ in $E_{2n}(\Pinfty)=\pi_{2n}(E\wedge\Pinfty)$.

When $E$ is orientable, the $\beta_n$ induce maps
\[    \xymatrix{
E\wedge S^{2n} \ar[r]^-{1\wedge\beta_n}& 
E\wedge E\wedge\Pinfty \to E\wedge\Pinfty,}
\]
and these maps induce an isomorphism from
$\bigoplus_{n=0}^{\infty}E\wedge S^{2n}$ to $E\wedge\Pinfty$
(by \cite[p.\,42]{A}). Consequently, we have
\[
E_*(\Pinfty)=\pi_*(E\wedge\Pinfty) \cong
\bigoplus_{n=0}^{\infty}\pi_*(E\wedge S^{2n}).
\]
Thus $E_*(\Pinfty)$ is a free graded $E_*$-module 
on generators $\beta_n$ in degree $2n$, $n\ge0$.

\smallskip
Let $\xi:S^{2}_+\cong \C\mathbb P^{1}_+\to\Pinfty$ 
be the map that classifies the tautological line bundle on
$\C\mathbb P^{1}$, and let 
$b:\Pinfty\wedge\spherespectrum\to S^{-2}\wedge\Pinfty\wedge\spherespectrum$
denote the adjoint of the stable map
\begin{equation*}
   \xymatrix{
S^2 \wedge\Pinfty\wedge\spherespectrum \ar[r]^-{\textrm{Hopf}} &
S^2_+\wedge\Pinfty\wedge\spherespectrum \;
\ar[r]^-{\xi\wedge1}& \;  \Pinfty\wedge\Pinfty\wedge\spherespectrum 
\ar[r]^-{m}&  \Pinfty\wedge\spherespectrum.}
\end{equation*}
Let $j=\xi-1$ denote the map $\Pinfty\to BU_+\to KU$
classifying the virtual tautological line bundle of rank $0$.
By Snaith \cite[2.12]{Sn}, the map $j$ induces an equivalence
from the homotopy colimit of the sequence
\begin{equation*}\label{KU.tower}
\Pinfty\wedge\spherespectrum \map{b} 
S^{-2}\Pinfty\wedge\spherespectrum\; \map{S^{-2}b\;} 
S^{-4}\Pinfty\wedge\spherespectrum \map{} \cdots.
\end{equation*}
to the spectrum $KU$. 
We write $u=u_{KU}$ for the map $S^2\to\Pinfty\map{j}KU$.
The map
\[
S^2\wedge KU \map{u\wedge1} KU\wedge KU \map{m} KU
\]
is the periodicity isomorphism for $KU$ (cf. \eqref{def:u}). 
Composing $j$ with the inverse of the periodicity isomorphism, we obtain 
an orientation for $KU$ whose formal group law is
$x+y+uxy$.

Smashing with $E$ yields maps
$b_E:E\wedge\Pinfty\to S^{-2}\wedge E\wedge\Pinfty$ 
and an equivalence between $E\wedge KU$ and 
the homotopy colimit of the sequence of spectra:
\begin{equation}\label{seq.spectra.1}
      \xymatrix{E\wedge \Pinfty \ar[r]^-{b_E}& 
   S^{-2}\wedge E\wedge \Pinfty \ar[r]^-{S^{-2} b_E}& 
   S^{-4}\wedge E\wedge \Pinfty \ar[r]& \cdots.}
\end{equation}

\smallskip
We now make the connection with the algebraic considerations
of Section \ref{sec:binomial}.

\begin{lem}\label{aijk}
The coefficient of $x^iy^j$ in $(x+y+uxy)^k$ is
$\binom{k}{k-i,k-j,i+j-k}u^{i+j-k}$.
\end{lem}

\begin{proof}
The product has $3^k$ terms. The only terms producing the monomial 
$x^iy^j$ are those with exactly $k-j$ of the factors must be $x$,
exactly $k-i$ must be $y$, and the remaining $i+j-k$ must be $uxy$.
\end{proof}

\begin{cor}\label{E.Pdot}
Suppose that $E$ is an oriented spectrum with a 
multiplicative group law $\mu(x,y)=x+y+uxy$ such that $u$ is a 
unit of $E_*$. Then 
the homotopy groups of the cosimplicial spectrum 
$E\wedge{\Pinfty}^{\wedge\mathdot+1}$ form a cosimplicial group
$E_*\oo F^{\oo\mathdot+1}\!,$ isomorphic to
the cobar complex $E_*\oo C_{F\oo F}^\mathdot(F,F)$ 
of Example \ref{triv.algebroid}. 
\end{cor}

\begin{proof}
By the above remarks, $\pi_*(E\wedge{\Pinfty}^{\wedge n+1})$ is a free 
$E_*$-module with basis $\beta_{i_0}\wedge\cdots\beta_{i_n}$, 
which we identify with $E_*\oo F^{\oo n+1}$ by sending $\beta_i$ to
$u^i\alpha_i$, as in \ref{variant:v=ut}. The coface maps are easily seen
to correspond, and for the codegeneracies it suffices to consider 
$\sigma^0:\pi_*(E\wedge\Pinfty\wedge\Pinfty)\to\pi_*(E\wedge\Pinfty).$
Adams showed in \cite[II(3.4--5)]{A} that the cofficient $a_{ij}^k$ of
$\sigma^0(\beta_i\oo\beta_j)$ is the coefficient of
$x^iy^j$ in $(x+y+uxy)^k$, which is given by Lemma \ref{aijk},
and agrees with $(u^i\alpha_i)\ast(u^j\alpha_j)$ by \eqref{alpha*alpha}.
\end{proof}

\begin{lem}\label{E.KU}
Suppose that $E$ is an oriented spectrum with a multiplicative group law
$\mu(x,y)=x+y+uxy$ such that $u$ is a unit of $E_*$. Then
\[
\pi_*(E\wedge KU)\cong\varinjlim \pi_*(S^{-2n}\wedge E\wedge\Pinfty)
\cong E_*\oo \varinjlim F \cong E_*\oo H.
\]
\end{lem}

\begin{proof}
Identifying $R=E_*$ and $F_R=E_*(\Pinfty)=\pi_*(E\wedge\Pinfty)$, 
Adams shows in  \cite[II(3.6)]{A} that
multiplication by $\beta_1$ is given by the formula \eqref{eq:T},
with $t$ replaced by $v=ut$ as in Variant \ref{variant:v=ut}.
%
Comparing with \eqref{seq:F}, we see that the ring 
$E_*(\Pinfty)[1/\beta_1]$ is isomorphic to $E_*\oo H$.
From \eqref{seq.spectra.1}, we obtain the result.
\end{proof}

\begin{subrem}
Lemma \ref{E.KU} fails dramatically if $u=0$, for example when $E=H\Z$.
\end{subrem}

\smallskip
Replacing $E$ by $E\wedge KU$, 
which we orient by $x\wedge\eta_{KU}$,
Lemma \ref{E.KU} yields $\pi_*(E\wedge KU\wedge KU)\cong E_*\oo H\oo H$.
The two maps $E_*\oo H\to E_*\oo H\oo H$, induced by
$E\wedge KU \to E\wedge KU^{\wedge 2}$, are given by $\eta_E\oo1$ and
$1\oo\Delta_H$, and $E_*\oo H\oo H \map{\sigma^0} E_*\oo H$ is given
by Corollary \ref{E.Pdot}. An inductive argument establishes: 

\begin{cor}\label{pi*E.KU(n)} 
The homotopy groups of the cosimplicial spectrum 
form the cosimplicial group:
\begin{equation*}
\pi_*(E\wedge KU^{\wedge n+1}) = E_*\oo H^{\oo n+1}.
\end{equation*}
This is the cobar complex $C_\Gamma^\mathdot(M,E_*)$ 
of Example \ref{RH.algebroid} with $M=E_*\oo H$.
\end{cor}

When $E=KU$, we can also form the cosimplicial spectrum 
$N^\mathdot(KU)$, as in the introduction, and we have:

\begin{prop}\label{NKU}
The cosimplicial ring $\pi_*(N^\mathdot KU)$ is
$KU_*\oo H^{\oo \mathdot}$, the cobar complex 
$C_\Gamma^\mathdot(KU_*,KU_*)$ of \ref{cobar}.
In particular, for every $n$ we have an isomorphism
\begin{equation*}\label{eq:ENdot}
\pi_*(KU^{\wedge n+1}) \cong KU_*\oo H^{\oo n}.
\end{equation*}
\end{prop}

\bigskip
{\it $KU^{\wedge\mathdot}$ with finite coefficients}

\medskip\noindent
Fix a prime $\ell$ and a power $q=\ell^\nu$, and let $P$ denote the 
Moore spectrum for $\Z/q$. 
We will assume that $q$ is chosen so that $P$ is a 
ring spectrum with unit. For example, if $\ell\neq 2,3$ then
this is the case for any $\nu\ge1$; see \cite[IV.2.8]{WK}.
Smashing $P$ with $KU$ yields the standard cosimplicial spectrum
$P\wedge KU^{\wedge n+1}$ associated to the triple $-\wedge KU$.
Set $R=\pi_*(P\wedge KU)$; the Universal Coefficient Theorem yields
$R=\Z/q[u,1/u]$. Since $P\wedge KU$ is oriented by the orientation on $KU$,
we see from Lemma \ref{E.KU} and Theorem \ref{H/q.basis} that 
$\pi_*(P\wedge KU\wedge KU)=R\oo H$ is a free $R$-module with basis
$\{\beta_n:n\not\equiv0\pmod\ell\}$, and from \eqref{pi*E.KU(n)} that 
$\pi_*(P\wedge KU^{\wedge n+1})=R\oo H^{\oo n}$.
In fact, $(R,R\oo H)$ is the mod $q$ reduction of the Hopf algebroid
of Lemma \ref{algebroid}.

Set $E=P\wedge KU$. Then each $\beta_n$ determines a map
$S^{2n}\wedge E\to E\wedge\Pinfty\to E\wedge KU$.
By Lemma \ref{E.KU}, the direct sum of these maps is a
homotopy equivalence of spectra
\[
\bigoplus E\wedge S^{2n} \map{\simeq} E\wedge\Pinfty
\]
There is a sequence \eqref{seq.spectra.1} for $E\wedge\Pinfty$, 
with colimit $E\wedge KU$.
By Example \ref{RH.algebroid}, $\pi_*(E\wedge KU)\cong E_*\oo H$.
Since $R=\pi_*(E)$ is a $\Z/q$-module, 
Theorem \ref{H/q.basis} applies: the
$\beta_n$ with $n\not\equiv0\pmod\ell$ form a basis of $E_*\oo H$.

\begin{prop}
For $E=P\wedge KU$, there is a homotopy equivalence of spectra
\[
\bigoplus_{n\not\equiv 0 \;\mathrm{mod}\;\ell} E\wedge S^{2n}
\map{\simeq} E\wedge KU.
\]
\end{prop}

\begin{proof}
The argument we gave  in Corollary \ref{R=Z/l}  shows that the 
left side is the colimit of \eqref{seq.spectra.1}. i.e., $E\wedge KU$.
In effect, we just notice that the components 
$E\wedge S^{2n}\to E\wedge S^{2n-2}$ of $b$ are multiplication by $n$
and are thus null-homotopic when $\ell|n$, while the maps
$E\wedge S^{2n}\to E\wedge S^{2n}$ of $b$ are multiplication by $(n+1)$
and are thus null-homotopic when $\ell|n+1$.
\end{proof}

\newpage 
\section{The motivic $E\wedge N^\mathdot(\Pinfty)$}\label{sec:KGL}

We now turn to motivic spectra. 
The goal of this section is to describe the cosimplicial spectrum
$E\wedge N^\mathdot(\Pinfty)= E\wedge{\Pinfty}^{\wedge\mathdot+1}$ 
associated to an oriented motivic spectrum $E$. 
Here the motivic space $\Pinfty$ is defined as the 
union of the spaces represented by the schemes $\mathbb P^n$. 

The Segre maps $\Proj^m\times\Proj^n\to\Pinfty$ induce a map
$m:\Pinfty\times\Pinfty\to \Pinfty$ making 
$\Pinfty\wedge\spherespectrum$ into a ring $T$-spectrum. 
Then $n\mapsto E\wedge{\Pinfty}^{\wedge n+1}$ 
is the standard cosimplicial spectrum associated to a
spectrum $E$ and the triple $-\wedge\Pinfty$ \cite[8.6.4]{WH}. 
If $E$ is a ring $T$-spectrum, then this is a
cosimplicial ring spectrum.

A commutative ring $T$-spectrum $E$ with 
unit $\spherespectrum\map{\eta}E$ is said to be {\it oriented} 
if there is a map $\Pinfty\wedge\spherespectrum\map{t} T\wedge E$ 
whose restriction to $\Proj^1\wedge\spherespectrum$ is $T\wedge\eta$.
(Compare with Section \ref{sec:KU}.)
As pointed out in Example 2.6.4 of \cite{PPR}, 
$KGL$ is an oriented ring spectrum in $\stablehomotopy$ with a 
multiplicative formal group law \cite[1.3.1 and 1.3.3]{PPR}.

When $E$ is oriented, the projective bundle theorem in \cite[3.2]{D}
yields an isomorphism with inverse $\oplus c^n$:
\begin{equation}\label{proj.decomp}
\oplus\beta_n:
\bigoplus_{n=0}^{\infty}E\wedge T^{n} \map{\simeq} E\wedge\Pinfty.
\end{equation}
If we let $F_\gr$ denote the free graded abelian group on generators 
$\beta_n$ in degree $n\ge0$, then $E\wedge\Pinfty$ is the free graded
$E$-module spectrum $E\oo F_\gr$.
By \eqref{proj.decomp}, the terms $E\wedge{\Pinfty}^{\wedge n+1}$ 
are isomorphic to the direct sum of terms
$E\wedge T^{i_0}\wedge\cdots\wedge T^{i_n}$, $i_j\ge0$.

\begin{lem}\label{coface.P}
The coface map $\partial^j:E\wedge{\Pinfty}^{\wedge n}\to
E\wedge{\Pinfty}^{\wedge n+1}$ sends the summand
$E\wedge T^{i_1}\wedge\cdots\wedge T^{i_n}$ to $E\wedge 
T^{i_1}\wedge\cdots T^{i_j}\wedge T^0\wedge T^{i_{j+1}}\cdots\wedge T^{i_n}.$
\end{lem}

\begin{proof}
Since the coface maps are determined by the unit map 
$\spherespectrum\to\Pinfty$, it suffices to observe that the map
$$E=E\wedge T^0\to E\wedge\Pinfty\cong\oplus E\wedge T^n$$
is just the inclusion of $T^0$ into $\oplus T^n$, 
by the projective bundle theorem.
\end{proof}

At this point, we may interpret the coface maps in 
$E\wedge N^\mathdot(\Pinfty)$ in terms of $F_\gr$. 
Recall (from \cite[8.1.9]{WH})
that a {\it semi-cosimplicial} object in a category is a sequence
of objects $K^n$ with coface operators $\partial^i:K^{n-1}\to K^n$
($0\le i \le n$) satisfying the cosimplicial identities
$\partial^j\partial^i=\partial^i\partial^{j-1}$ if $i<j$.

For example, $n\mapsto F_\gr^{\oo n+1}$ is a semi-cosimplicial
graded abelian group, where the codegeneracies insert $1=\beta_0$.
(The generators $\beta_n$ of $F_\gr$ lie in degree $n$.)
We may also form the semi-cosimplicial spectrum $E\oo F_\gr^{\oo n+1}$.
Note that the codegeneracy maps $F_\gr^{\oo n+1}\to F_\gr^{\oo n}$
associated to the product $F\oo F\to F$
are not graded, as \eqref{alpha*alpha} shows, so we do not get a
cosimplicial spectrum in this way.

\begin{cor}
Let $E$ be an oriented spectrum $E$. Then there is an isomorphism
of semi-cosimplicial ring spectra,
$E\oo F_\gr^{\oo\mathdot+1} \map{\cong}E\wedge{\Pinfty}^{\wedge\mathdot+1}$
\end{cor}


\begin{lem}\label{lem.diag}
For any oriented $E$,
the component maps $E\wedge T^n\to E\wedge T^{i}\wedge T^j$ of the
diagonal $E\wedge\Pinfty\to E\wedge\Pinfty\wedge\Pinfty$ are 
the canonical associativity isomorphisms when $i+j=n$, 
and zero otherwise.
\end{lem}

\begin{proof}
The maps $\beta_n:E\wedge T^n\to E\wedge\Pinfty$ in \eqref{proj.decomp}
satisfy $c^k\circ\beta_n =\delta_{kn}$, where $c^n$
is the projection $E\wedge\Pinfty\to E\wedge T^n$. 
We have a commutative diagram:
\[\xymatrix{
E\wedge T^n \ar[r]^-{\beta_n}&
E\wedge\Pinfty \ar[r]^-{\Delta}\ar[d]^{\Delta}& 
E\wedge(\Pinfty)^{\wedge i+j} \ar[r]^-{c,c,...,c}\ar[d]_{\cong}& 
E\wedge T^{\wedge i+j} \ar[d]^-{\cong}\\ &
E\wedge\Pinfty\wedge\Pinfty \ar[r]^-{\Delta}&
E\wedge(\Pinfty)^{\wedge i}\wedge(\Pinfty)^{\wedge j} \ar[r]^-{c,...,c}&
E\wedge T^{\wedge i}\wedge T^{\wedge j}.}
\]
The top horizontal composite is $c^{i+j}\circ\beta_n=\delta_{n,i+j}$,
and the entire composition is the component map in question.
The result follows.
\end{proof}
%


We now turn to the codegeneracies 
$\sigma^j: E\wedge{\Pinfty}^{\wedge n+1}\to E\wedge{\Pinfty}^{\wedge n}$.
These are all induced from the product $\Pinfty\wedge\Pinfty\to\Pinfty$.
After taking the smash product with $E$, and using \eqref{proj.decomp},
we may rewrite the product as an $E$-module map
\[
\bigoplus_{i,j\ge0}E\wedge T^i\wedge T^j = 
(E\wedge\Pinfty)\wedge(E\wedge\Pinfty) \to E\wedge\Pinfty = 
\bigoplus_{k\ge0} E\wedge T^k.
\]
Since each term $E\wedge T^i\wedge T^j$ is a compact $E$-module,
its image lies in a finite sum. Thus the product is given by a matrix
whose entries $a_{ij}^k$ are the component maps
$E\wedge T^i\wedge T^j\to E\wedge T^k$\!.


In terms of motivic homotopy groups,
$E_{*,*}(\Pinfty)\cong\bigoplus_{n=0}^{\infty}E_{*,*}(T^{n})$
is a free graded $E_{*,*}$-module with generators 
$\beta_n\in E_{2n,n}(\Pinfty)$, $n\ge0$, corresponding to the unit of
$E_{2n,n}(T^n)\cong E_{0,0}$. Dually, $E^{*,*}(\Pinfty)\cong E^{*,*}[[x]]$
and (as in topology) the orientation yields a 
formal group law of the form $\mu=x+y+\sum_{i,j\ge1}a_{i,j}x^{i}y^{j}$
in $E_*[[x,y]]\cong E^*(\Pinfty \times \Pinfty)$, where
\begin{equation}\label{FGL.coef}
	a_{i,j}\in E_{2i+2j-2,i+j-1} = \Hom_{\SH}(T^{i+j},E\wedge T).
\end{equation}

We can now extend a result of Adams to the motivic setting.

\begin{thm}\label{Adams.a1jk}
If $E$ is an oriented motivic ring spectrum, then
$a_{1j}^k = k\; a_{1,1+j-k}$, with $a_{1,1+j-k}$ as in \eqref{FGL.coef}.
\end{thm}

\begin{proof}
This follows directly from Lemma \ref{lem.diag}, 
together with Adams' calculation in \cite[II(3.6), p.\,46]{A}.
\end{proof}

\begin{subrem}
For fixed $i,j$ the maps $a_{ij}^k$ are zero unless $k\le i+j$.
Therefore the image of $E\wedge T^i\wedge T^j$ under the product
lands in the finitely many terms $E\wedge T^k$ with $k\le i+j$.
\end{subrem}


Our next task is to replace $E\oo F_\gr^{\oo\mathdot+1}$ by 
$E\oo F^{\oo\mathdot+1}$ to get a cosimplicial spectrum.
This is possible when $E$ is periodic in the following sense.
If $E_*$ has a unit in $E_{2,1}$, represented by a map
$T\map{u} E$, then $u\wedge E$ induces an isomorphism
$T\wedge E\to E$ with inverse $u^{-1}\wedge E$. We call it a 
{\it periodicity map} for $E$, and say that $E$ is {\it periodic}.
If $E$ is oriented and periodic, we can  use the periodicity map 
for $E$ to define maps $\alpha_n:E\to E\wedge\Pinfty$
such that $\beta_n=u^n\alpha_n$, resulting in a rewriting of the
projective bundle formula  \eqref{proj.decomp} as
\begin{equation}\label{proj.periodic}
\oplus\alpha_n: \bigoplus\nolimits_{n=0}^\infty\; E 
\map{\cong} E\wedge\Pinfty.
\end{equation}
Using this new basis, each $E\wedge{\Pinfty}^{\wedge n}$ is isomorphic to
$E\oo F^{\oo n}$.

Now suppose that $E$ is an oriented spectrum with a multiplicative 
formal group law $\mu=x+y+uxy$, with $u$ a unit in $E_{2,1}$.
In this case, it is convenient to change our orientation to eliminate $u$,
as suggested in \cite[II(2.1)]{A}. This produces the new formal group law
$x+y+xy$. To see this, let $t'$ denote the element $ut$ of $E^0(\Pinfty)$;
then the formal group law implies that
$$\mu(t')=(ux)+(uy)+(ux)(uy)=x'+y'+x'y'.$$


\begin{prop}\label{toy.global}
Let $E$ be an oriented ring spectrum $E$ with multiplicative group law
$x+y+uxy$, $u$ a unit. Then the cosimplicial ring spectrum
$E\wedge{\Pinfty}^{\wedge\mathdot+1}$ is $E\oo C^\mathdot_{F\oo F}(F,F)$.

In particular, $KGL\wedge{\Pinfty}^{\wedge\mathdot+1}$ is isomorphic
to $KGL\oo C^\mathdot_{F\oo F}(F,F)$.
\end{prop}

\begin{proof}
After the change in orientation indicated above, the projection bundle 
formula yields isomorphisms
$E\oo F^{\oo n} \map{\cong} E\wedge{\Pinfty}^{\wedge n}$.
The respective coface maps are insertion of $1$ and $T^0$ 
(by Lemma \ref{coface.P}), so they agree. 
The codegeneracies of the left side are given by the product
$F\oo F\to F$, whose coefficients are given by \eqref{alpha*alpha}.
These agree with the matrix entries $a_{ij}^k$ on the right
by Lemma \ref{aijk} and Corollary \ref{E.Pdot}, 
as the proof of \ref{E.Pdot} shows.
\end{proof}

\section{The slice filtration}\label{sec:slice}

Recall \cite[2.2]{Vslice} that $s_q(E)$ denotes the $q^{th}$ slice of 
a motivic spectrum $E$, and that $s_q(E\wedge T^i)=s_{q-i}(E)\wedge T^i$. 
 The following result is well known to experts.

\begin{prop}\label{slice.sum}
The slice functors $s_q$ commute with direct sums.
\end{prop}

\begin{proof}
It suffices to prove that $f_{q}$ commutes with direct sums, 
since there is a distinguished triangle of the form
$f_{q+1}\rightarrow f_{q} \rightarrow s_{q}.$
For simplicity, we will restrict ourselves to the case 
$q=0$; the argument for arbitrary $q$ is exactly the same.

Consider a direct sum $E=\oplus_{\alpha}E_{\alpha}$ in $\mathcal{SH}$.  
Since $\eff$ is closed under direct sums, $\oplus f_0(E_\alpha)$
is effective. By \cite[3.1.14]{Pelaez}, there is a family of compact
objects $K$ such that, for every $X\to Y$, $f_0(X)\to f_0(Y)$ 
is an isomorphism if and only if each $[K,X]\to[K,Y]$ is an isomorphism.
In the case at hand, for every such $K$ we have
\[
[K,\oplus f_0(E_\alpha)]\cong\oplus[K,f_0(E_\alpha)]\cong
                         \oplus[K,E_\alpha]\cong[K,\oplus E_\alpha].
\]
Therefore $\oplus f_0(E_\alpha)\to f_0(\oplus E_\alpha)$ is an isomorphism.
\end{proof}

\begin{exam}\label{slice.Pinfty}
Suppose that $E$ is oriented.  
Applying $s_q$ to \eqref{proj.decomp} yields the formula
$\oplus s_{q-n}(E)\wedge T^n \map{\cong}s_q(E\wedge\Pinfty)$
with component maps $s_q(\beta_n)$.
More generally, the slice 
$s_q(E\wedge{\Pinfty}^{\wedge n})$ is isomorphic to 
\[
\bigoplus s_{r}(E)\wedge T^{i_1}\wedge\cdots\wedge T^{i_n},
\]
where $q=r+i_1+\cdots+i_n$,
and the coface maps in $s_q(E\wedge{\Pinfty}^{\wedge\mathdot})$ 
are given by Lemma \ref{coface.P}, with $E$ replaced by $s_q(E)$.
Thus, except for the codegeneracy maps, the cosimplicial spectrum
$\oplus_q s_q(E\wedge{\Pinfty}^{\wedge\mathdot+1})$ is isomorphic to
$\oplus s_q(E) \oo F_\gr^{\oo\mathdot+1}$.
\end{exam}

\smallskip
We will need the following observation: for any motivic ring spectrum
$E$, $s_*E=\oplus s_q(E)$ is a graded motivic ring spectrum;
see \cite[3.6.13]{Pelaez}).  
Since $\Z[u,u^{-1}]$ is a graded ring (with $u$ in degree 1), 
we can form the graded motivic ring spectrum 
$s_0(E)\oo\Z[u,u^{-1}] = \bigoplus T^q\wedge s_0(E)$.

\begin{lem}\label{Kslice.ring}
As ring spectra,
$s_0(KGL)\oo\Z[u,u^{-1}] \map{\cong}\oplus s_q(KGL)$.

More generally, $s_0 E\oo\Z[u,u^{-1}] \map{\cong}s_*E$ 
for any oriented ring spectrum $E$ with a unit $u$ in $E_{2,1}$.
\end{lem}

\begin{proof}
The periodicity isomorphisms $u^q:KGL\wedge T^q\to KGL$, and
more generally $E\wedge T^q\to E$, induce isomorphisms 
$s_0 E\wedge T^q\cong s_q E$ compatible with multiplication. 
These assemble to give the result. (Cf.\ \cite[6.2]{Sp}.)
\end{proof}

\begin{prop}\label{toy.slice}
Suppose that $E$ is oriented, and has a multiplicative group law $x+y+uxy$.
If $u$ is a unit, we have isomorphisms of cosimplicial ring spectra:
\[
s_*(E\wedge{\Pinfty}^{\wedge\mathdot+1}) \cong
s_*(E) \oo F^{\oo\mathdot+1} \cong
s_0(E) \oo\Z[u,u^{-1}]\oo C^\mathdot_{F\oo F}(F,F).
\]
In particular,
$s_*(KGL\wedge{\Pinfty}^{\wedge\mathdot+1}) \cong
 s_0(KGL)\oo\Z[u,u^{-1}]\oo C^\mathdot_{F\oo F}(F,F).$
\end{prop}

\begin{proof}
Applying $s_q$ to \ref{proj.periodic} yields the simple formula,
with component maps $s_q(\alpha_n)$:
\[
s_q(E)\oo F =
\bigoplus\nolimits_{n=0}^\infty\; s_q(E) \map{\cong} s_q(E\wedge\Pinfty).
\]
The rest is immediate from Proposition \ref{toy.global}
and Example \ref{slice.Pinfty}.
\end{proof}

When $E=KGL$, Propositions \ref{toy.global} and \ref{toy.slice} give
the formulas for $KGL\wedge{\Pinfty}^{\wedge\mathdot+1}$ 
and its slices which were mentioned in the Introduction.

\smallskip
\section{$E\wedge KGL$}\label{sec:E.K}

In this section, we describe the augmented cosimplicial spectrum
$E\wedge KGL^{\wedge\mathdot+1}$ associated to an oriented spectrum $E$
and the triple $-\wedge KGL$. 
\[
E \map{\eta_L} E\wedge KGL \rightrightarrows E\wedge KGL\wedge KGL 
\threerightarrows E\wedge KGL^{\wedge3} \cdots
\]
Replacing $E$ by $E\wedge KGL^{\wedge n}$,
we are largely reduced to the description of $E\wedge KGL$, $\eta_L$
and the map $\sigma^0:E\wedge KGL\wedge KGL\to E\wedge KGL$.

We begin with a few generalities.
The $T$-spectrum $KGL$ comes with a periodicity isomorphism
$T\wedge KGL\to KGL$; see \cite[6.8]{V-ICM} and \cite[3.3]{Vslice};
our description of it is taken from \cite[1.3]{PPR}.
Let $\xi:T \to\Pinfty$ be the map that 
classifies the tautological line bundle on $\mathbb P^{1}$, and let 
$b:\spherespectrum\wedge\Pinfty \to T^{-1}\wedge\spherespectrum\wedge\Pinfty$
denote the adjoint of the map
\begin{equation*}
   \xymatrix{
   T\wedge\Pinfty\;     \ar[r]^-{\xi\wedge1}& \;
   \Pinfty\wedge\Pinfty \ar[r]^-{m}& \Pinfty.}
\end{equation*}
As observed by Gepner-Snaith \cite[4.17]{GS} and 
Spitzweck-{\O}stv{\ae}r \cite{SO}, 
$KGL$ is the homotopy colimit of the resulting sequence,
\begin{equation*}
\xymatrix{\spherespectrum\wedge \Pinfty \ar[r]^-{b}& 
T^{-1}\wedge \spherespectrum\wedge \Pinfty \ar[r]^-{T^{-1}b}& 
T^{-2}\wedge \spherespectrum\wedge \Pinfty \ar[r]^-{T^{-2}b}& \cdots,}
\end{equation*}
and the colimit map $\spherespectrum\wedge\Pinfty\to KGL$ is the map
classifying the virtual tautological line bundle $\xi-1$ of degree~0.
We write $u_K$ for the map 
$T\map{}\spherespectrum\wedge\Pinfty\map{\xi-1}KGL$. 
Then multiplication by $u_K$ is the periodicity isomorphism
\begin{equation}\label{def:u} 
\xymatrix{
T\wedge KGL \ar[r]^-{u_K\wedge1}& KGL\wedge KGL \map{m} KGL}
\end{equation}
for the $T$-spectrum $KGL$; see \cite[1.3.3]{PPR}.  

Smashing $b$ with any spectrum $E$ yields maps 
$b_E:E\wedge\Pinfty\to T^{-1}\wedge E\wedge\Pinfty$, and
yields 
a sequence of $T$-spectra 
\begin{equation}\label{seq.spectra.3}
\xymatrix{ E\wedge \Pinfty \ar[r]^-{b_E}& 
T^{-1}\wedge E\wedge \Pinfty \ar[r]^-{T^{-1}b_E}& 
T^{-2}\wedge E\wedge \Pinfty \ar[r]^-{T^{-2}b_E}& \cdots,}
\end{equation}
with homotopy colimit $E\wedge KGL$, parallel to the sequence
\eqref{seq.spectra.1} in topology.

When $E$ is oriented, the adjoint $E\wedge T\wedge\Pinfty\to E\wedge\Pinfty$
of $b_E$ is the composition of the map
$\beta_1: E\wedge T\to E\wedge\Pinfty$ of \eqref{proj.decomp}, 
smashed with $\Pinfty$, with the product on $\Pinfty$.
Therefore $b_E$ corresponds to multiplication by $\beta_1$, and
the following observation follows from Theorem \ref{Adams.a1jk}.

\begin{lem}\label{Adams.bis}
If $E$ is oriented,
the components $E\wedge T^j\to T^{-1}\wedge E\wedge T^k$
of the map $b_E$ of \eqref{seq.spectra.3}
are multiplication by $k\ a_{1,1+j-k}$.
\end{lem}

We now assume that $E$ has a multiplicative formal group law 
$\mu=x+y+uxy$ with $u\in E_{2,1}$ a unit of $E_*$. If 
$\alpha_n:E\to E\oo F\cong E\wedge\Pinfty$
is the map defined by \eqref{proj.periodic}, 
Lemma \ref{Adams.bis} says that
\[
b_E(\beta_n) = n\, u\, \beta_n + (n+1)\beta_{n+1}, \quad\text{or}\quad
b_E(\alpha_n) = n\,\alpha_n + (n+1)\alpha_{n+1}.
\]
Comparing with \eqref{eq:T}, 
this shows that the map 
$b_E: E\wedge\Pinfty\to T^{-1}\wedge E\wedge\Pinfty$ is induced from
the homomorphism $T:F\to F$, $T(f)=tf$. 
Recall from Definition \ref{def:H} 
that the Hopf algebra $H$ is defined by $H=F[1/t]$ and that 
$C_H^n(\Z,\Z)=H^{\oo n}$ in the classical cobar complex 
(see Example \ref{triv.algebroid}).

\begin{thm}\label{EK.toy}
If $E$ has a multiplicative group law $x+y+uxy$ with $u$ a unit,
we have $E\wedge KGL \cong E\oo H$,
and isomorphisms of cosimplicial spectra for all $q\in\Z$:
\begin{align*}
E\wedge KGL^{\wedge\mathdot+1} \cong\ & E\oo H^{\oo\mathdot+1} \cong
E\oo C^\mathdot_{H}(\Z,\Z);
\\
s_*(E\wedge KGL^{\wedge\mathdot+1}) \cong\  &
s_*(E) \oo H^{\oo\mathdot+1} \cong
s_*(E) \oo C^\mathdot_{H}(\Z,\Z).
\end{align*}
In addition, $E\to E\wedge KGL$ corresponds to $E\oo\Z\to E\oo H$.
\end{thm}

\begin{proof}
By \eqref{eq:T}, multiplication by $t$ on $E\oo F$ is given by
the same formula as multiplication by $b_E$ on $E\wedge\Pinfty$.
It follows that the homotopy colimit $E\wedge KGL$ of the sequence
\eqref{seq.spectra.3} is the same as 
$E\oo \text{colim}(F\map{t}F\to\cdots) = E\oo H$. Replacing $E$ by
$E\wedge KGL^{\wedge n}$ shows that 
$E\wedge KGL^{\wedge n+1} \cong E\oo H^{\oo n+1}$ and hence that
$s_*(E\wedge KGL^{\wedge n+1}) \cong s_*(E)\oo H^{\oo n+1}$.
The coface and codegeneracies of these cosimplicial spectra are 
identified by Propositions \ref{toy.global} and \ref{toy.slice}.
\end{proof}

When $E=KGL$ we get the formula mentioned in the Introduction:

\begin{cor}\label{K.toy}
$KGL\wedge KGL \cong KGL\oo H$, and 
$KGL^{\wedge\mathdot+2} \cong KGL\oo H^{\oo\mathdot+1}$.
In addition, $\eta_L:KGL\to KGL\wedge KGL$
corresponds to $KGL\oo\Z\to KGL\oo H$.
\end{cor}

\begin{exam}\label{Q.toy}
When $E=E\oo\Q$, we have 
$E\wedge KGL^{\wedge n}\cong E\oo\Q[t_1,t_1^{-1},...,t_n,t_n^{-1}]$,
because $\Q\oo H=\Q[t,t^{-1}]$.
\end{exam}

%

\medskip
Set $q=\ell^\nu$ and write $E/q$ for the cofiber $E\oo\Z/q$ of $E\map{q}E$.
An elementary calculation shows that $(E/q)\oo H\cong E\oo H/qH$.

\begin{cor}\label{E.K mod-l}
Suppose that $E$ has a multiplicative formal group law 
$x+y+uxy$, such that  $u$ is a unit of $E_{*,*}$. Then the
family of maps $\beta_n:E\wedge T^{n}\to E\wedge\Pinfty\to E\wedge KGL$
with $n\not\equiv0\pmod{\ell}$
induces an isomorphism
\[ 
E\oo H/\ell^\nu H \cong
\bigoplus_{n\not\equiv0} (E/\ell^\nu)\wedge T^{n}
\map{\cong} (E/\ell^\nu)\wedge KGL.
\]
Moreover, the slices $s_m(E/\ell^\nu\wedge KGL)$ are isomorphic to 
$\bigoplus_{n\not\equiv0}\, s_{m-n}(E/\ell^\nu)\wedge T^n$.
\end{cor}

\begin{proof}
Immediate from Theorems \ref{H/q.basis} and  \ref{EK.toy}.
\end{proof}

%

\section{The cosimplicial $KGL$ spectrum}\label{sec:final}

In this section, we complete the information in the previous
section to determine the slices of the cosimplicial spectrum
$N^\mathdot KGL$, with $N^nKGL=KGL^{\wedge n+1}$. 
\[
KGL \rightrightarrows KGL\wedge KGL \threerightarrows
KGL^{\wedge3} \fourrightarrows\cdots
\]
Its codegeneracies are given by the product on $KGL$, 
and its coface maps are given by insertion of the unit 
$\eta_K:\spherespectrum\to KGL$; in particular, 
$\partial^0,\partial^1:KGL\to KGL\wedge KGL$ are the canonical maps 
$\eta_R=\eta_K\wedge1$ and $\eta_L=1\wedge\eta_K$, respectively. 
There is also an involution $c$ of $KGL\wedge KGL$ swapping the 
two factors, and we have $\eta_R=c\circ\eta_L$. 


By Corollary \ref{K.toy},
there are isomorphisms $KGL\oo H^{\oo n}\map{\cong} KGL^{\wedge n+1}$ 
such that the diagram
\begin{equation*}
\xymatrix{
KGL \oo\Z\ar[r]^-{1\wedge\text{incl}}\ar[d]^-{\cong}& 
KGL\oo H\ar[d]^-{\cong} \\
KGL \ar[r]^-{\eta_L} & KGL\wedge KGL
}\end{equation*}
commutes. As $\partial^1=\eta_L$, 
this establishes the initial case of the following lemma.

\begin{lem}\label{not.partial0}
The isomorphisms $KGL\oo H^{\oo n}\map{\cong} KGL^{\wedge n+1}\!$
of Corollary \ref{K.toy}
are compatible with all the coface and codegeneracy operators
of $N^\mathdot KGL$, except possibly for $\partial^0$ and $\sigma^0$.
\end{lem}

\begin{proof}
Recall the standard construction (the dual path space of
\cite[8.3.14]{WH}) which takes a cosimplicial object $X^\mathdot$
and produces a new cosimplicial object $Y^\mathdot$ with
$Y^n=X^{n+1}\!,$ coface maps $\partial_Y^i=\partial_X^{i+1}$ and
codegeneracy maps $\sigma_Y^i=\sigma_X^{i+1}$. Applying this
construction to $N^\mathdot KGL$ yields the cosimplicial spectrum
$KGL^{\wedge\mathdot+2}$ of Corollary \ref{K.toy}. Since this is
isomorphic to the cosimplicial spectrum $KGL\oo H^{\oo\mathdot}$
by Theorem \ref{EK.toy}, the result follows.
\end{proof}

\begin{proof}[Proof of Theorem \ref{thm:main}(a)]
Taking slices in Lemma \ref{not.partial0}, we see that the isomorphisms
$s_*(KGL)\oo H^{\oo n} \map{\cong} s_*(KGL^{\wedge n+1})$
are compatible with all of the coface and codegeneracy operators,
except possibly for $\partial^0$ and $\sigma^0$.
\end{proof}

\begin{rmk}
It is tempting to consider the cosimplicial spectrum
$KGL\oo C^\mathdot_H(\Z,\Z)$, where $C^\mathdot_H(\Z,\Z)$
is the cobar complex over $H$. However, its coface maps $\partial^0$ and
codegeneracy maps $\sigma^0$ are not the same as in $N^\mathdot KGL$.
\end{rmk}

To finish the proof of Theorem \ref{thm:main}(b), we need to show that
the isomorphims in Proposition \ref{not.partial0} are indeed compatible
with $\partial^0$ and $\sigma^0$, under the two additional assumptions.
We will do so in Propositions \ref{partial0} and \ref{sigma0}

\medskip
\begin{defn}
The map $v_K:T\wedge\spherespectrum \to KGL\wedge KGL$ is defined
to be $\eta_R(u_K)$, where $u_K:T\to KGL$ is the map of \ref{def:u}.

We define $v: T\wedge KGL\wedge KGL \to KGL\wedge KGL$
to be multiplication by $v_K$.
\end{defn}

\begin{lem}\label{etaR+uv}
We have commutative diagrams
\begin{equation*}
\xymatrix{
T\wedge KGL \ar[r]^-{T\wedge \eta_R} \ar[d]^-{u}&
T\wedge KGL \wedge KGL \ar[d]^-{v} \\
KGL \ar[r]^-{\eta_R} & KGL \wedge KGL
}
\quad
\xymatrix{
T\wedge s_q(KGL) \ar[r]^-{T\wedge \eta_R} \ar[d]^-{u}&
T\wedge s_q(KGL \wedge KGL) \ar[d]^-{v} \\
s_{q+1}(KGL) \ar[r]^-{\eta_R} & s_{q+1}(KGL \wedge KGL).
}\end{equation*}
\end{lem}

\begin{proof}
For reasons of space, we write $K$ for $KGL$. In the diagram below,
the top rectangle commutes by definition of $v_K$, and the bottom
rectangle commutes because $\eta_R$ is a morphism of ring spectra.
\begin{equation*}
\xymatrix{
T\wedge\spherespectrum  \ar[r]^-{T\wedge\eta_K}\ar[ddr]_-{u_K}&
T\wedge K \ar[r]^-{T\wedge\eta_R} \ar[d]^-{u_K\wedge K} &
T\wedge (K\wedge K) \ar[d]_-{v_K\wedge K\wedge K} &
T\wedge\spherespectrum \ar[l]_-{T\wedge\eta_{(K\wedge K)}} \ar[ddl]^-{v_K}
\\
&K\wedge K \ar[d]^-{m_K} \ar[r]^-{\eta_R\wedge\eta_R} & 
(K\wedge K)^{\wedge2} \ar[d]_-{m_{K\wedge K}}\\
&K \ar[r]^-{\eta_R} & K\wedge K.
}\end{equation*}
The left and right verticals are the maps $u$ and $v$, respectively.
This establishes commutativity of the first square in Lemma \ref{etaR+uv}. 
As the second square is the composition of the natural isomorphism
$T\wedge s_q(E)\map{\cong} s_{q+1}(T\wedge E)$ and the $q^{th}$ slice of
the first square, it also commutes.
\end{proof}

Recall that $s_0(\spherespectrum)\to s_0(KGL)$ is an isomorphism
if the base $S$ is smooth over a perfect field $k$, or any scheme over
a field of characteristic~0. This is so because it is true over
any perfect field \cite{L}, and $s_0$ commutes with the pullback
$\SH(k)\to\SH(S)$ in these cases; see 
\cite[2.16]{P-functorial}           
and \cite[3.7]{P-functorial}. 

\begin{prop}\label{partial0}
Assume that $s_0(\spherespectrum)\to s_0(KGL)$ is an isomorphism.
Then the isomorphisms $s_q(KGL)\oo H^{\oo n}\map{\cong} s_q(KGL^{\wedge n+1})$
are compatible with $\partial^0$.
\end{prop}

\begin{proof}
Since $\spherespectrum$ is an initial object, there is a unique
map from $\spherespectrum$ to $KGL\wedge KGL$; 
the assumption implies that $\eta_L=\eta_R$ as maps
$s_0(KGL) \to s_0(KGL\wedge KGL)$. Now for any $q\in\Z$,
Lemma \ref{Kslice.ring} says that $u^q:T^q\wedge s_0(KGL)\map{\cong}s_q(KGL)$.
By Lemma \ref{etaR+uv}, we have a diagram
\begin{equation*}
\xymatrix{
T^q\wedge s_0(KGL) \ar[r]^-{T^q\wedge \eta_R} \ar[d]^-{u^q}_{\cong}&
\; T^q\wedge s_0(KGL \wedge KGL) \ar[d]^-{v^q} \\
s_{q}(KGL) \ar[r]^-{\eta_R} & s_{q}(KGL \wedge KGL).
}\end{equation*}
This determines the maps $\eta_R:s_{q}(KGL)\to s_{q}(KGL \wedge KGL)$.
Summing over $q$, and invoking Lemma \ref{Kslice.ring} twice,
we see that $\eta_R$ is the map 
\[
s_0(KGL)\oo\Z[u,u^{-1}] \to s_0(KGL)\oo\Z[u,u^{-1},v,v^{-1}]
\]
sending the copy of $s_0(KGL)$
indexed by $u^q$ to the copy of $s_0(KGL)$ indexed by $v^q$.
Since we saw in Example \ref{RH.algebroid} that this is the same as
the map induced from
$\eta_R:\Z[u,u^{-1}]\to\Z[u,u^{-1}]\oo H$, 
this shows that we have a commutative diagram
\begin{equation*}
\xymatrix{
\bigoplus s_q(KGL) \ar[r]^-{1\oo\eta_R} \ar[d]^-{=} & 
\bigoplus s_q(KGL)\oo H \ar[d]^-{\cong} \\
\bigoplus s_q(KGL) \ar[r]^-{\partial^0}& \bigoplus s_q(KGL\wedge KGL).
}\end{equation*}
The result follows for $n>0$, since
$\partial^0:KGL^{\wedge n}\to KGL^{\wedge n+1}$ is
$\eta_R\wedge KGL^{\wedge n-1}$.
\end{proof}

\medskip

To conclude the proof of Theorem \ref{thm:main}, we have to compare 
the slices of the codegeneracy $\sigma^0:KGL\wedge KGL\to KGL$ (the product)
with the map of Example \ref{RH.algebroid},
\begin{equation}\label{sigma0.target}
s_0(KGL)\oo\Z[u,u^{-1}]\oo H \map{\varepsilon} s_0(KGL)\oo\Z[u,u^{-1}],
\quad  \varepsilon(u)=\varepsilon(v)=u.
\end{equation}
Recall from \ref{def:H} that $\Z[t,t^{-1}]$ is a subring of $H$, and that
$\varepsilon:H\to\Z$ sends $t$ to $1$.

\begin{lem}\label{sigma0.t}
The restriction of the product 
$s_*(KGL\wedge KGL)\to s_*(KGL)$ to
\[
s_*(KGL)\oo\Z[t,t^{-1}] \to s_*(KGL)\oo H \cong  s_*(KGL\wedge KGL)
\]
equals the composition 
\[
s_*(KGL)\oo\Z[t,t^{-1}] \map{\varepsilon}s_*(KGL)\oo\Z
\cong s_*(KGL).
\]
\end{lem}

\begin{proof}
Since $\sigma^0:KGL\wedge KGL\to KGL$ is a left inverse to both
$\eta_L$ and $\eta_R$, and $u=\eta_L(u)$, $v=\eta_R(u)$, we see that
$\sigma^0$ sends the copies of $s_0(KGL)$ indexed by the monomials
$u^q$ and $v^q$ (in $\Z[u,u^{-1}]\oo H$) to the copies indexed by $u^q$.
Since $\sigma^0$ is a map of ring spectra, the copies indexed by 
$u^i\oo t^j=u^{i-j}v^j$ map to the copies indexed by $u^{i}$.
\end{proof}

Recall from Example \ref{Q.toy} that 
$KGL^{\wedge n+1}\oo\Q\cong KGL\oo\Q[t_0,t_0^{-1},...,t_n,t_n^{-1}]$.
This is $KGL\oo C^\mathdot_{\Q[t,t^{-1}]}(\Z,\Z)$,
where $C^\mathdot_{\Q[t,t^{-1}]}(\Z,\Z)$ is
the $n^{th}$ term of the cobar complex
of Example \ref{RH.algebroid} for the Hopf algebra $\Q[u,u^{-1}]$.

Now consider $(R,\Gamma)$ with $R=\Q[u,u^{-1}]$ and $\Gamma=R[t,t^{-1}]$,
as in Example \ref{RH.algebroid}. Combining Lemma \ref{not.partial0}, 
Proposition \ref{partial0} and Lemma \ref{sigma0.t}, we obtain the
rational version of the KGL slice conjecture \ref{conjecture8}:

\begin{cor}
Let $(R,\Gamma)$ be as above, and assume that
$s_0(\spherespectrum)\to s_0(KGL)$ is an isomorphism.
As cosimplicial motivic ring spectra,
\[
s_* N^\mathdot(KGL\oo\Q) \cong s_0(KGL)\oo 
C^\mathdot_{\Gamma}(R,R).
\]
\end{cor}

\bigskip
As in \cite[6.1]{AHS}, it is easy to verify that the diagram 
\begin{equation}
\xymatrix{
KGL\wedge\Pinfty \ar[r]^-{\xi-1} \ar[d]^-{b_K}& KGL\wedge KGL \ar[d]^-{v}\\
T^{-1}\wedge KGL\wedge\Pinfty \ar[r]^-{\xi-1} & T^{-1}\wedge KGL\wedge KGL}
\end{equation}
commutes, where $b_K$ is $KGL\wedge b$, $b$ is 
the map in \eqref{seq.spectra.3} and the right side is 
multiplication by $v_K$ (smashed with $T^{-1}$).

\begin{lem}\label{s0(1)=HZ}
If $S$ is smooth over a perfect field, then 
$s_0(\spherespectrum)\to H[\Z]$ is an isomorphism in $\SH(S)$.
\end{lem}

\begin{proof}
This is true over the ground field $F$, by \cite[11.3.6]{L}.
Now the slice functors 
$s_q$ commute with the pullback $\pi^*$
over $\pi:S\to\Spec(F)$; see \cite[2.16]{P-functorial}.
Because $S$ is smooth, $H\Z$ also pulls back over $\pi^*$ 
(see \cite[6.1]{V-ICM}, \cite[3.18]{V-EMsp}), and we have 
\[
s_0(\spherespectrum_S)=\pi^*s_0(\spherespectrum_F) =
\pi^*(H\Z_F) = H\Z_S.
\qedhere\]
\end{proof}

Now the group $\Hom_{\SH(S)}(H\Z,H\Z)\cong H^0(S,\Z)$ is torsionfree,
and equal to $\Z$ if $S$ is connected; cf.\,\cite[3.7]{V-RPO}. 
It follows that if $S$ is 
smooth over a perfect field, so that $s_0(\spherespectrum)\cong H\Z$,
the the hypothesis of the following Proposition is satisfied.

\begin{prop}\label{sigma0}
Assume that $\Hom_{\SH(S)}(s_0(\spherespectrum),s_0(\spherespectrum))$
is torsionfree.  Then the isomorphisms 
$s_q(KGL)\oo H^{\oo n}\map{\cong} s_q(KGL^{\wedge n+1})$
are compatible with $\sigma^0$.
\end{prop}

\begin{proof}
Every element of $H$ is a power of $t$ times a numerical polynomial,
so it suffices to check the test maps
$\alpha_n:s_0(K)\to s_0(KGL)\oo H\cong s_0(KGL\wedge KGL)$ 
corresponding to the elements $\binom{t}{n}$ of $H$. 
If $n=0,1$ we are done by Lemma \ref{sigma0.t}.
For $n\ge2$ the composition
\[ 
s_0(KGL) \map{n!} s_0(KGL) \map{\alpha_n} s_0(KGL\wedge KGL) 
\map{\sigma^0} s_0(KGL)
\]
is given by $\sigma^0(t(t-1)\cdots(t-n+1))=0$ in the ring
$[s_0(KGL),s_0(KGL)]$, which is $\Z$ by Lemma \ref{s0(1)=HZ}. 
Since $\sigma^0\alpha_n$ corresponds to
an integer, and is killed by $n!$, we must have $\sigma^0\alpha_n=0$.
\end{proof}



\section*{Acknowledgements}

The authors were supported by NSF grant FRG0966824.
They are grateful to Oliver R\"ondigs and Paul {\O}stv{\ae}r
for pointing out the calculations in Adams' book \cite{A},
to Markus Spitzweck for conversations about \cite{Sp},
and to all three for useful discussions about the slice conjectures.
%


\end{document}